\documentclass[10pt, twoside]{amsart}
\usepackage{amsmath,mathtools}
\usepackage{amsmath, amsthm, amssymb, amsfonts, enumerate}
\usepackage[colorlinks=true,linkcolor=blue,urlcolor=blue]{hyperref}
\usepackage{dsfont}
\usepackage{geometry}
\usepackage{todonotes}
\usepackage{epstopdf}
\usepackage{bbm}
\usepackage{geometry}
\usepackage[utf8]{inputenc}
\usepackage{bm}
\usepackage{amsfonts}
\usepackage{amsfonts}
\usepackage{textcomp}
\usepackage{amssymb}
\usepackage{float}
\usepackage{tikz}
\usepackage{epsfig}
\usepackage{amsmath}
\usepackage[english]{babel}
\usepackage{a4}
\usepackage{enumerate}
\usepackage{tcolorbox}
\usepackage{soul}
\newcommand{\stkout}[1]{\ifmmode\text{\sout{\ensuremath{#1}}}\else\sout{#1}\fi}

\newtheorem{theorem}{Theorem}[section]

\newtheorem{remark}[theorem]{Remark}
\newtheorem{hypothesis}[theorem]{Assumption}

\newtheorem{prop}[theorem]{Proposition}
\newtheorem{coroll}[theorem]{Corollary}
\newtheorem{definition}[theorem]{Definition}

\newcommand{\ds}{\displaystyle}
\newcommand{\tr}{\mathop{\mathrm{Tr}}\nolimits}

\def \R{\mathbb{R}}

\usepackage[T1]{fontenc}

\title[Stochastic  control with measurable coefficients]{Stochastic Optimal Control  with Measurable Coefficients  and Applications}

\date{\today}
\author[de Feo]{Filippo de Feo}\address{F.~de Feo:  Institut für Mathematik, Technische Universität Berlin, Berlin, Germany and Department of Economics and Finance,  Luiss Guido Carli University, Rome, Italy.}\email{\href{mailto:defeo@math.tu-berlin.de}{defeo@math.tu-berlin.de}}

\numberwithin{equation}{section}

\begin{document}

\begin{abstract}
We consider fully non-linear stochastic optimal control problems with measurable coefficients and (local) uniformly elliptic diffusion.   Using the theory of $L^p$-viscosity solutions, we show existence of an $L^p$-viscosity solution $v\in W_{\rm loc}^{2,p}$ of the Hamilton-Jacobi-Bellman (HJB) equation, which, in turn, is also a strong solution (i.e. it satisfies the HJB equation pointwise a.e.). We are then led to prove verification theorems, providing necessary and sufficient conditions for optimality. These results allow us to  construct optimal feedback controls and to characterize the value function as the unique $L^p$-viscosity solution of the HJB equation. To the best of our knowledge, these are the first results for fully non-linear stochastic optimal control problems with measurable coefficients. We use the theory developed to solve a stochastic optimal control problem arising in economics within the context of optimal advertising.
\end{abstract}

\maketitle
{\bf Mathematics Subject Classification (2020):} 49L25, 93E20, 49K45,  49L20,  49L12, 49N35 

{\bf Keywords:}  stochastic optimal control, measurable coefficients, Hamilton-Jacobi-Bellman equation, $L^p$-viscosity solution, viscosity solution, verification theorem, optimal synthesis, economic models, optimal advertising models  

\tableofcontents
\section{Introduction}
Consider a fully non-linear stochastic control problem with infinite horizon in which the goal is to minimize, over all admissible controls $u(\cdot)$, a cost functional of the form
\begin{equation}\label{eq:functional_intro}
  J(x,u(\cdot)) :=
\mathbb E\left[\int_0^{\infty} e^{-\rho t}
l(y(t),u(t)) dt \right],
\end{equation}
where the state $y(t)$ is subject to the dynamics
\begin{equation*}
dy(t) = \ds b \left ( y(t),u(t) \right)
         dt 
\ds  + \sigma \left (y(t), u(t) \right)\, dW(t), \quad \forall t \geq 0,\quad y(0)=x  \in \mathbb{R}^n.
\end{equation*}
\paragraph{\textbf{Literature review.}}
Stochastic optimal control problems  have been deeply studied under standard continuity assumptions (e.g. global continuity plus Lipschitz continuity in $x$  of $b(x,u),\sigma(x,u)$, uniformly in $u$, and local uniform continuity in $x$ of $l(x,u)$, uniformly in $u$) using various approaches (see e.g. \cite{fleming_soner,krylov,nisio,pham_2009,touzi,yong_zhou}). 
 One of the most successful {approaches}  is the one via viscosity solutions of the  Hamilton-Jacobi-Bellman (HJB) equation, which is a fully non-linear second-order partial differential equation of the form
\begin{equation*}
\rho v -H(x,Dv,D^2v)=0,
\quad  x \in \mathbb R^n.
\end{equation*}
This approach allows a comprehensive understanding of the control problem, e.g. it is possible to characterize the value function as the unique solution of the HJB equation, prove verification theorems as well  as construct optimal feedback controls (see e.g. \cite{fleming_soner,nisio,pham_2009,touzi,yong_zhou,gozzi_swiech_zhou_2005,gozzi_swiech_zhou_2010}).

However, to our knowledge,  control problems measurable coefficients (in the state and control variables) have not been  well studied, although being particularly relevant as real world applications may exhibit irregular dynamics. The only previous  paper on the topic seems to be \cite{menoukeu_tangpi}, where an approach via maximum principle has been employed {for a drift of the form $b(t, x, a)=b_1(t, x)+b_2(t, x, a)$, with $b_1$ bounded and measurable (but independent of the control) and $b_2$  smooth in ($x, a$),} and where the noise is additive and non-degenerate {(see Theorem 1.1 there)}. This  framework is motivated by an optimal consumption problem in which investors are subject to a wealth tax. 

When the coefficients are merely measurable the standard approach via viscosity solutions of HJB equations is not possible. However, the notion of $L^p$-viscosity solution (recalled in Appendix \ref{sec:appendix_definitions}) has been introduced in the literature to treat second-order fully non-linear uniformly elliptic and parabolic partial differential  equations (PDEs) with measurable coefficients and a very powerful theory has been developed (see e.g. \cite{caffarelli_c_s,crandall_kocan_soravia_swiech,crandall_kocan_swiech,koike_perron,swiech1997})\footnote{see also \cite{zhang_zhu_zhu} for a different approach for HJB equations  with distribution-valued coefficients}. In this definition the $C^2$ test functions of viscosity solutions are replaced with the larger class of $W^{2,p}$-test functions, so that a stronger notion of solution is obtained, {allowing to achieve  uniqueness of $L^p$-viscosity solutions for PDEs with measurable coefficients on bounded domains   and to prove powerful $W^{2,p}$-regularity results \cite{crandall_kocan_swiech,swiech1997}\footnote{
This theory was  used in \cite{defeo_federico_swiech}  to get $W_{\rm loc}^{2,p}$ and $C^{1,\alpha}$ partial regularity for the viscosity solution of an HJB equation on an (infinite-dimensional)  Hilbert space related to a stochastic control problem with delays (coefficients there satisfy standard Lipschitz continuity assumptions and diffusion is non-degenerate) and the  $C^{1,\alpha}$ partial regularity was used in a semilinear case in \cite{deFeoSwiech} to construct optimal feedback controls  (see also \cite{defeo_phd}).}.}

\paragraph{\textbf{Our results.}} It is evident from the literature review above that control problems with measurable coefficients are not well understood.
The goal of this paper is to tackle  the fully non-linear stochastic optimal control problem \eqref{eq:functional_intro} (introduced in Section \ref{sec:formul})
with measurable coefficients and (local) uniform elliptic diffusion. 
In particular, in  Section \ref{sec:HJB}, using the theory of $L^p$-viscosity solutions \cite{caffarelli_c_s,swiech1997}, we show that if $v$ is an $L^p$-viscosity solution with $p \geq n$, then $v\in W_{\rm loc}^{2,p}$,  so that it is also a strong solution of the HJB equation (i.e. it satisfies HJB pointwise a.e.) (Proposition \ref{th:existence_uniqueness_Lp_viscosity}). Moreover, if the Hamiltonian $H$ is also continuous, this is true for any viscosity solution (Proposition \ref{th:existence_uniqueness_viscosity_infinite}). Using  Perron's method \cite{koike_perron}, we give sufficient conditions for the existence of an $L^p$-viscosity solution $v$ for $p \geq n$ (Proposition \ref{prop:existence_viscosity}). 

Motivated by the previous points, 
in Section \ref{sec:verification} we consider a $W_{\rm loc}^{2,n}$-strong solution $v$ of the (fully non-linear) HJB equation. Using Dynkin's formula for $W_{\rm loc}^{2,n}$-functions \cite{krylov} (recalled here in Appendix \ref{app:ito}), we prove verification theorems, providing  necessary and sufficient conditions for optimality (Theorems \ref{th:verification}, \ref{th:verification_II},  \ref{th:verification_necessary} and Remark \ref{rem:strong_solutions}). As a consequence, such verification theorems hold when $v$ is an $L^p$-viscosity solution. If the Hamiltonian $H$ is also continuous, they also hold when $v$ is a (standard) viscosity solution (Corollaries \ref{cor:verification_viscosity_sufficient}, \ref{cor:verification_theorem_necessary_corollary}). In  these cases, our optimal control $u^*(\cdot)$ satisfies the relation 
\begin{equation*}
 u^*(s)\in  {{\rm argmin}}_{u \in U}   H_{\rm cv}(y^*(s),D v(y^*(s))  ,D^2 v(y^*(s))  ,u), \quad   \mathbb P\textit{-a.s., for  a.e. } s \geq 0,
 \end{equation*}
 where $H_{\rm cv}$ is the current value Hamiltonian and $Dv, D^2v$ are defined a.e.\footnote{of course, if $v \in W^{2,p}_{\rm loc}(\mathbb R^n)$ for $p >n$, by Sobolev embedding theorems, the first order gradient $Dv$ is well-defined and continuous in the classical sense.}
We also remark, {as we will see in Section \ref{sec:synthesis}}, that the sufficient verification theorem Theorem \ref{th:verification} (together with Remark \ref{rem:strong_solutions}) and Corollary \ref{cor:verification_viscosity_sufficient} can be used to obtain uniqueness of strong solutions  or  $L^p$-viscosity solutions of the HJB equation (i.e. from the statements of the theorems we get that $v$ is equal to the value function).

In Section \ref{sec:synthesis}, we use our verification theorems to solve the stochastic control problem by constructing optimal feedback controls  (Propositions \ref{cor:closed_loop},  \ref{prop:optimal_feedback_example}) and  {by characterizing $V$ as the unique $L^p$-viscosity solution and unique strong solution of the HJB equation, see Corollary \ref{coroll:existence_uniqueness}.} 

\paragraph{\textbf{Comparison with the literature.}} {To the best of our knowledge, this is the first paper that addresses fully non-linear stochastic optimal control problems with measurable coefficients.}  {We remark that our results on uniqueness of $L^p$-viscosity solutions on unbounded domains are also particularly relevant since we are not aware of general uniqueness results for such solutions of PDEs on unbounded domains with measurable coefficients. We achieve these  via a control-verification technique, which is different from an analytic method used in  \cite{crandall_kocan_swiech,swiech1997} for PDEs  on bounded domains.}

We also stress that our results complement the existing ones for non-smooth solutions of the HJB equation, which are available when the coefficients satisfy standard Lipschitz continuity assumptions:  for instance, with respect to \cite{gozzi_russo} we treat fully non-linear stochastic control problems, and we do not assume the semiconcavity of $v$ as in \cite{yong_zhou,gozzi_swiech_zhou_2005,gozzi_swiech_zhou_2010} (see, in particular \cite{gozzi_swiech_zhou_2010})\footnote{see also, e.g., \cite{defeo_phd,deFeoSwiech,defeo_swiech_wessels,fgs_book,federico_gozzi_2018,stannat_wessels_2021} and the references therein for results of this kind in infinite-dimension.}.

{Subsequent to the submission of this manuscript, two further works appeared on arXiv studying stochastic control problems with measurable coefficients, reflecting the growing interest in this area. In particular, \cite{du_wei} studies finite horizon stochastic optimal control problems for SDEs whose drift and running
cost are merely measurable in the state variable and the noise is additive, using PDE techniques and a policy iteration scheme. 
Moreover, \cite{criens} (in particular version v4, i.e. arXiv:2404.17236v4, submitted to arXiv in September 2025, while versions v1-v3  treated continuous coefficients) studies the case of measurable coefficients with $L^p$-drift and uniformly elliptic diffusion using a semigroup approach and the theory of $L^p$-viscosity solutions.}

\paragraph{\textbf{Applications to optimal advertising.}}
In Section \ref{sec:application}, using the theory developed, we solve a stochastic optimal control  problem arising in economics within the context of optimal advertising (see e.g. \cite{grosset_viscolani,nerlove_arrow,deFeo_2023,defeo_federico_swiech,gozzi_marinelli_2004}  and the references therein).
We consider a firm who seeks to optimize the advertising campaign for a certain product. We assume that the stock of the advertising goodwill $y(t)$ of the product is given by the following controlled $1$-dimensional SDE
  \begin{equation*} 
 dy(t) = \left[ a (y(t))+c u(t) \right]          dt   + [\nu(y(t)) +\gamma u(t)] dW(t),\quad   y(0)=x,  
 \end{equation*} 
 where  the control $u(t)$ is the rate of the investment; $c> 0$; $a(x)$ is a non-positive bounded measurable function (hence, it is allowed to be discontinuous),  representing image deterioration under different regimes, depending on the level of the goodwill; the real valued Brownian motion $W(t)$ represents the uncertainty in the market; $\gamma>0$ and $\nu(x)$ represent the intensity of the uncertainty in the model. Here,  $\nu$ satisfies suitable assumptions.   The goal is to minimize
 $$ \mathbb{E} \left[\int_0^\infty e^{-\rho s} (h(u(s))-g(y(s))) d s\right], $$ where  $\rho >0$ is a discount factor,  $h(u)$  is a cost function  and $g (x)$ is an  utility function, which is also allowed to be discontinuous.
 Using the theory developed in the previous sections, we completely solve the optimal advertising problem by characterizing the value function as the unique $L^p$-viscosity solution of the HJB equation and by constructing optimal feedback controls.

Throughout the whole paper, we will use the notations from Appendix \ref{sec:notations}. 
\section{The stochastic optimal control problem}\label{sec:formul}
In this section, we introduce the stochastic optimal control problem.

We start by recalling the concept of  generalized reference probability spaces and reference probability spaces.
\begin{definition}\label{def:generalized_ref_prob_space}
A generalized reference probability space is $\eta=(\Omega,\mathcal{F},\mathcal{F}_t,
W(t),\mathbb{P})_{t\geq 0}$, where  $(\Omega,\mathcal{F},\mathbb{P})$ is a complete  probability space equipped with a {complete right-continuous} filtration $\mathcal F_t$  and $W(t)$ is a standard $\mathbb{R}^n$-valued Wiener process; $\eta$ is called a reference probability space when $\mathcal F_t$ is the augmented filtration generated by $W$.
\end{definition}
\begin{remark}\label{rem:properclass_ref_prob_spaces}
We will work with the class of all (generalized) reference probability spaces, which\footnote{similarly to the proper class of all sets} is a proper class (i.e. not a set, in the sense of Von Neumann–Bernays–Gödel set theory) \cite{deFeoSwiech}.
\end{remark}

We introduce an infinite horizon optimal control problem in the weak formulation. 
On some generalized reference probability space $\eta=(\Omega,\mathcal{F},\mathcal{F}_t,W(t),\mathbb{P})_{t\geq 0}$, we consider the following controlled stochastic differential equation (SDE) 
\begin{equation}
\label{eq:SDDE}
dy(t) = \ds b \left ( y(t),u(t) \right)
         dt 
\ds  + \sigma \left (y(t), u(t) \right)\, dW(t), \quad \forall t \geq 0,\quad y(0)=x  \in \mathbb{R}^n,
\end{equation}
where $b \colon \mathbb{R}^n \times U \to \mathbb{R}^n$, $\sigma \colon \mathbb{R}^n  \times U \to \mathbb R^{n\times n}$, with $U$ {being a non-empty Borel subset of $\mathbb R^n$}, and  $u(\cdot): \Omega\times [0,+\infty )\to U$ is a control process forming part of an admissible control in $\mathcal U$ defined below.  
 
We consider a cost functional of the form
\begin{equation}
\label{eq:obj-origbis}
J(x,u(\cdot)) :=
\mathbb E\left[\int_0^{\infty} e^{-\rho t}
l(y(t),u(t)) dt \right],
\end{equation}
where $\rho>0$ is the discount factor and
$l \colon \mathbb{R}^n \times U \to \mathbb{R} $ is the running cost.

We will consider the following assumptions.
\begin{hypothesis}\label{hp:measurability}
$b,\sigma,l$ are Borel measurable. 
 \end{hypothesis}
    \begin{hypothesis}\label{hp:locally_bounded_coeff}
 $b(\cdot,u),\sigma(\cdot,u),l(\cdot,u)$ are bounded on bounded subsets of $\mathbb R^n$, uniformly in $u \in U$.
 \end{hypothesis}
 For some results, we will strengthen the previous assumption to the following ones:
     \begin{hypothesis}\label{hp:b,sigma,bounded}
 $b,\sigma$ are bounded.
 \end{hypothesis}
 \begin{hypothesis}\label{hp:l_subpolynomial_growth}
 There exist $C>0 ,m \geq 0$  s.t. $|l(x,u)| \leq C(1+|x|^m)$, for all  $x \in \mathbb R^n, u \in U$.
  \end{hypothesis}
      \begin{hypothesis}\label{hp:l,bounded}
 $l$ is bounded.
 \end{hypothesis}
 \begin{hypothesis}\label{hp:sigma_uniformly_continuous}
 $\sigma(\cdot,u)$ is uniformly continuous on bounded sets of $\mathbb R^n$, uniformly in $u \in U.$
 \end{hypothesis}
 
 Important assumptions will be local uniform elliptic condition and uniform elliptic condition below.
 \begin{hypothesis}\label{hp:uniform_ellipticity}
For every $R>0$  there exists $\lambda_{R} >0$ such that  $ \sigma(x,u) \sigma(x,u)^T  \geq \lambda_R I,$ for all $x\in B_R, u \in U.$
\end{hypothesis}
 \begin{hypothesis}\label{hp:uniform_ellipticity_strong}
Let $\sigma \colon \mathbb{R}^n  \times U \to S^n$ such that there exists $\lambda >0$ such that $ \sigma(x,u)  \geq \lambda I ,$ for all $x,  u \in U.$
\end{hypothesis}

Next, we define   admissible controls.
\begin{definition}\label{def:weak-form} 
We say that $(\Omega,\mathcal{F},\mathcal{F}_t,u(t),y(t),
W(t),\mathbb{P})_{t\geq 0}$ is an admissible control at $x\in\mathbb R^n$ if   $\eta=(\Omega,\mathcal{F},\mathcal{F}_t,
W(t),\mathbb{P})_{t\geq 0}$ is generalized reference probability space, $u(\cdot): \Omega\times [0,+\infty )\to U$ is $\mathcal{F}_t$-progressively measurable, and $y(t)$ is an $\mathcal F_t$-adapted solution of \eqref{eq:SDDE} on $\eta$ with $y(0)=x$ and with continuous paths a.s. such that
\begin{align*}
&  \mathbb E\left[\int_0^{\infty} e^{-\rho t} \left | l(y(t),u(t)) \right |  dt \right] < \infty. 
\end{align*}
  The pair $(y(\cdot),u(\cdot))$ is called an admissible pair (for the initial state $x$). We denote by $\mathcal{U}$ the class of admissible controls at the fixed initial state $x$, suppressing its dependence on $x$. With an abuse of notation we will simply write $u(\cdot)\in\mathcal U$ for the entire tuple above.
\end{definition}
\begin{remark}
As  members of $\mathcal{U}$ range over the proper class of all generalized reference probability spaces $\eta$ (Remark \ref{rem:properclass_ref_prob_spaces}), $\mathcal{U}$ may be a proper class.
\end{remark}
Throughout the whole paper, given an admissible pair $(y(\cdot),u(\cdot))$, for $R>0 ,t\geq 0$, we will define the stopping time
 \begin{equation}\label{eq:def_stopping_time}
 \tau^R := \inf \{s \in[0, t]:|y(s)|>R\}.
 \end{equation}

We impose the next two assumptions here in order to have a well-defined optimal control problem without requiring additional regularity  on the coefficients (but they will be dropped in Section \ref{sec:synthesis}). 
\begin{hypothesis}\label{hp:state-eq}
We assume that $\mathcal U\neq\emptyset$ for every initial state $x\in\mathbb R^n$.
\end{hypothesis}

The goal is to minimize $J(x,u(\cdot)) $ over all admissible controls $u(\cdot) \in \mathcal{U}$.
Following the dynamic programming approach, we define the value function for the optimal control problem by
$$
V: \mathbb R^n \rightarrow [-\infty,+\infty], \quad V(x):=\inf _{u(\cdot) \in \mathcal{U}} J(x , u(\cdot)) \quad \forall x \in \mathbb R^n.
$$
\begin{hypothesis}\label{hp:V}
 $V$ is well-defined (i.e. $V(x)>-\infty$ for every $x \in \mathbb R^n$).
\end{hypothesis}
Assumptions \ref{hp:state-eq} and \ref{hp:V} ensure that $V$ is finite-valued. In statements concerning a fixed initial state $x$, admissibility refers to that state and these assumptions are required only at $x$.
An admissible control $u(\cdot)\in \mathcal U$ is said to be optimal if  $J(x,u(\cdot)) = V(x)$.

\section{$L^p$-viscosity solutions of HJB equations}\label{sec:HJB}
In  this section, we study the HJB equation using the notion of $L^p$ viscosity solutions. Using the theory of $L^p$-viscosity solutions \cite{caffarelli_c_s,swiech1997}, if $v$ is an $L^p$-viscosity solution with $p\geq n$, we prove that $v\in W_{\rm loc}^{2,p}$,  so that it is also a strong solution of the HJB equation. Moreover, if the Hamiltonian $H$ is also continuous, this is true for any viscosity solution. Using  Perron's method \cite{koike_perron}, we give sufficient conditions for the existence of an $L^p$-viscosity solution $v$ for $p \geq n$. 
These facts justify an approach via verification theorems to the control problem in the next sections. We will use the notations and  definitions from Appendices \ref{sec:notations}, \ref{sec:appendix_definitions}. 
\subsection{$L^p$-viscosity solutions}
We define the current value Hamiltonian  $H_{\rm cv}:\mathbb R^n \times \mathbb R^n\times   S^n \times U \to \R$  and the Hamiltonian
 $H:\mathbb R^n \times \mathbb R^n\times   S^n \to \R$, respectively,  by 
 \begin{small}
\begin{align*}
&H_{\rm cv}(x,p,Z,u):= b(x,u) \cdot p+ {1\over 2}\tr (\sigma(x,u)\sigma(x,u)^T Z) + l(x,u) ,  & \forall x,p \in \mathbb R^n, Z \in S^{n}, u \in U;\\
& H(x,p,Z) :=  \inf_{u\in U}  H_{\rm cv}(x,p,Z,u) ,  & \forall x,p \in \mathbb R^n, Z \in S^{n}.
\end{align*}
 \end{small}
 \begin{hypothesis}\label{hp:H_measurable}
There exists a countable subset $\overline U \subset U$  such that $H(x,p,Z) =  \inf_{u\in \overline U}  H_{\rm cv}(x,p,Z,u),$ for all $ x,p \in \mathbb R^n, Z \in S^{n}$.
\end{hypothesis}
\begin{remark}
\begin{enumerate}[(i)]
\item  If Hypotheses \ref{hp:measurability}, \ref{hp:H_measurable} are satisfied, then $H$ is Borel measurable.
\item Hypothesis \ref{hp:H_measurable} is satisfied, e.g. if $b(x,\cdot), \sigma(x,\cdot), l(x,\cdot)$ are continuous for every $x$.
\end{enumerate}
\end{remark}

 The HJB equation associated with the optimal control problem is the following second order fully non-linear partial differential equation
\begin{equation}
\label{eq:HJB}
\rho v - H(x,Dv,D^2v)=0,
\quad x \in \mathbb R^n.
\end{equation}
In order to use the theory of $L^p$-viscosity solutions, we rewrite \eqref{eq:HJB} in the form
$$F(x,v,Dv,D^2v)=f(x), \quad  x \in
\mathbb R^n,$$
where  $F \colon \mathbb R^n \times  \mathbb R \times \mathbb R^n \times S^n \to \mathbb R,$ $ F(x,r,p,X):=\rho r - H(x,p,X)+H(x,0,0),$ $f\colon \mathbb R^n  \to \mathbb R,$ $f(x):=H(x,0,0)$, for all $x,p \in \mathbb R^n,r \in \mathbb R, X \in  S^n$.
\begin{prop}\label{th:existence_uniqueness_Lp_viscosity}
Let  Assumptions \ref{hp:measurability}, \ref{hp:locally_bounded_coeff}, \ref{hp:sigma_uniformly_continuous}, \ref{hp:uniform_ellipticity}, \ref{hp:H_measurable} hold. 
\begin{enumerate}[(i)]
\item Let  $v \in C(\mathbb R^n)$ be an $L^p$-viscosity solution of \eqref{eq:HJB} for some $p\geq n$. Then, we have $v\in W^{2,p}_{\rm loc}(\mathbb R^n)$  and it is a strong solution of  \eqref{eq:HJB}.
\item Conversely, let $p \geq n$ and  $v\in W^{2,p}_{\rm loc}(\mathbb R^n)$ such that it is a strong solution of  \eqref{eq:HJB}.  Then $v$ is an   $L^p$-viscosity solution of \eqref{eq:HJB}.
\end{enumerate}
\end{prop}
\begin{proof}
\begin{enumerate}[(i)]
\item  
Let $R>0$ be arbitrary; then:
\begin{itemize}
\item $F$ satisfies \cite[Structure condition (SC)]{caffarelli_c_s} on $B_R \times  \mathbb R \times \mathbb R^n \times S^n$  thanks to Assumptions \ref{hp:locally_bounded_coeff}, \ref{hp:uniform_ellipticity}.
\item $F$ is convex in $p,X$;
\item condition \cite[Eq. (3.1)]{swiech1997} is satisfied  due to Assumption \ref{hp:locally_bounded_coeff}, \ref{hp:sigma_uniformly_continuous}, \ref{hp:uniform_ellipticity}  (see \cite[Proof of Theorem 9.1]{crandall_kocan_swiech} for a detailed explanation). 
\item Due to Assumption \ref{hp:locally_bounded_coeff}, $f \in L^\infty(B_R)$.
\item  $v$ is also  an $L^p$-viscosity solution of 
\begin{equation}\label{eq:HJB_F_on_BR}
F(x,v,Dv,D^2v)=f(x),
\quad x \in B_R.
\end{equation}
\end{itemize}
Hence, we can apply  \cite[Theorem 3.1]{swiech1997}, to have $v\in W^{2,p}(B_R)$.

The fact that $v$ is a strong solution to the HJB equation follows by \cite[Theorem 3.6, Corollary 3.7]{caffarelli_c_s}.
\item By Sobolev embeddings, we have  $v \in   C^{0,\alpha}_{\rm loc}(\mathbb R^n)$ so that it is continuous. By \cite[Lemma 2.6]{caffarelli_c_s}, for every $R>0$, $v$ is an $L^p$-viscosity solution of \eqref{eq:HJB_F_on_BR}. The claim follows.
\end{enumerate}
\end{proof}
\begin{prop}\label{prop:existence_viscosity}
Let  Assumptions \ref{hp:measurability}, \ref{hp:locally_bounded_coeff}, \ref{hp:l,bounded},  \ref{hp:uniform_ellipticity_strong}, \ref{hp:H_measurable} hold and let $p \geq n$. Then there exists an $L^p$-viscosity solution $v \in C_b(\mathbb R^n)$ to \eqref{eq:HJB}.
\end{prop}
\begin{proof}
Let $C>0$ such that $|l(x,u)|\leq C$. Then $\underline v(x):=-C/\rho$, $\overline v(x):=C/\rho$ is a (classical) subsolution,  supersolution to \eqref{eq:HJB}, respectively. Then, setting 
$$
v(x):=\sup _{u \in \mathcal{S}} u(x),\quad  x \in \mathbb R^n, \quad \mathcal{S}:=\left\{ 
u \in C(\mathbb R^n) :
u \text { is an } L^p \text {-viscosity subsolution of  }\eqref{eq:HJB}, \underline{v} \leq u \leq \bar{v} \right\},
$$ 
which is a bounded function, we can apply \cite[Theorems 3.3, 4.1]{koike_perron} to have that $v$ is an $L^p$-viscosity solution of \eqref{eq:HJB}.
\end{proof}

\subsection{Viscosity solutions}
When $H$ is continuous, we will consider (standard) viscosity solutions.
 \begin{hypothesis}\label{hp:H_continuous}
Let $H$ be continuous.
\end{hypothesis}
\begin{prop}\label{th:existence_uniqueness_viscosity_infinite}
Let  Assumptions \ref{hp:measurability}, \ref{hp:locally_bounded_coeff}, \ref{hp:sigma_uniformly_continuous}, \ref{hp:uniform_ellipticity},  \ref{hp:H_continuous} hold. 
\begin{enumerate}[(i)]
\item Let  $v \in C(\mathbb R^n)$ be a  viscosity solution of \eqref{eq:HJB}. Then, for every $p\geq n$, we have $v\in W^{2,p}_{\rm loc}(\mathbb R^n)$  and it is a strong solution of  \eqref{eq:HJB}.
\item Conversely, let $p\geq n$ and  $v\in W^{2,p}_{\rm loc}(\mathbb R^n)$ such that it is a strong solution of  \eqref{eq:HJB}. Then $v$ is a  viscosity solution of \eqref{eq:HJB}. 
\end{enumerate}
\end{prop}
\begin{proof}
\begin{enumerate}[(i)]
\item  Due to Assumption \ref{hp:H_continuous}, we have that $F,f$  are continuous. Then, by Proposition \ref{prop:equivalence_viscosity_Lpviscosity}, for every $p\geq n$, the function $v$ is also  an $L^p$-viscosity solution of \eqref{eq:HJB_F_on_BR}. Therefore, we can use  Proposition \ref{th:existence_uniqueness_Lp_viscosity} to have the first claim.
\item by Proposition \ref{th:existence_uniqueness_Lp_viscosity}, we have  $v \in C(\mathbb R^n)$ is an $L^p$-viscosity solution of \eqref{eq:HJB}. Then, it follows that $v$ is a  viscosity solution of \eqref{eq:HJB}.
\end{enumerate}
\end{proof}


\section{Verification theorems}\label{sec:verification}
In the previous section,  we saw that, using the theory of $L^p$-viscosity solution, we can construct an $L^p$-viscosity solution of  the HJB equation  $v \in W^{2,n}_{\rm loc}$ (so that it is twice-differentiable a.e.), which is then also a strong solution.
Motivated by these facts, in this section we assume to be given  a strong (sub/super)solution $v \in W^{2,n}_{\rm loc}$ of the HJB equation; using  Dynkin's formula for $ W^{2,n}_{\rm loc}$-functions  (Theorem \ref{th:dynkin_formula}  \cite{krylov}), we prove  verification theorems providing necessary and sufficient conditions for optimality (i.e. Theorems \ref{th:verification}, \ref{th:verification_II}, \ref{th:verification_necessary}, together with Remark \ref{rem:strong_solutions}). As a consequence, such results hold for viscosity solutions or  $L^p$-viscosity solutions of the HJB equation (Corollaries \ref{cor:verification_viscosity_sufficient}, \ref{cor:verification_theorem_necessary_corollary}).    We will discuss in Remark \ref{rem:strong_solutions} the key role of sets of measure zero, where, contrary to the classical $C^2$-case, $v $ may not be differentiable or twice-differentiable.

In the following, given an admissible pair $(y(\cdot),u(\cdot))$, recall the definition of the stopping time $\tau^R$, $R>0$ given in \eqref{eq:def_stopping_time}. 
 \begin{theorem}[Verification, sufficient conditions for optimality]\label{th:verification} Let  Assumptions  \ref{hp:measurability}, \ref{hp:locally_bounded_coeff},  \ref{hp:uniform_ellipticity}, \ref{hp:state-eq}, \ref{hp:V}, \ref{hp:H_measurable} hold. Let  $v\in W^{2,n}_{\rm loc}(\mathbb R^n)$. Fix an initial state $x \in \mathbb R^n$ and let $(y(\cdot),u(\cdot))$ be an admissible pair. Consider the following conditions:
  \begin{equation}\label{eq:behaviour_Ev(y)}
  \begin{aligned}
  & \mathbb E \left[\sup_{0\leq r\leq s}|v(y(r))|\right] < \infty,\quad \forall s >0, \quad  \lim_{s \to \infty} e^{-\rho s} \mathbb E \left[  \left | v(y(s) )\right | \right] =0,
 \end{aligned}
   \end{equation}
\begin{align}
&\rho v(y(s)) -H(y(s),Dv(y(s)),D^2v(y(s))) \leq 0,  & \mathbb P\textit{-a.s., for  a.e. } s \geq 0,\label{eq:HJB(y(t))leq0}\\
& \rho v(y(s)) -H(y(s),Dv(y(s)),D^2v(y(s))) \geq 0, & \mathbb P\textit{-a.s., for  a.e. } s  \geq 0,\label{eq:HJB(y(t))geq0}\\
& \rho v(y(s)) - H(y(s),Dv(y(s)),D^2v(y(s))) = 0, & \mathbb P\textit{-a.s., for  a.e. } s  \geq 0. \label{eq:HJB(y(t))=0}
\end{align}
 Then the following statements hold:
 \begin{enumerate}[(i)]
 \item Let \eqref{eq:behaviour_Ev(y)} and  \eqref{eq:HJB(y(t))leq0}  hold for every admissible pair  $(y(\cdot),u(\cdot))$;  then  
 \begin{equation}\label{eq:v_leq_V}
 v(x)\leq V(x) .
 \end{equation}
 \item Let  \eqref{eq:HJB(y(t))geq0}  hold for some admissible pair $(y(\cdot),u(\cdot))$; then, for every $R>0$,
  \begin{equation} \begin{aligned}\label{eq:v_geq_fundamental_rel}
v( x) &  \geq   \mathbb E \left[ e^{-\rho (t\wedge \tau^R)}  v(y(t \wedge \tau^R )) \right] +\mathbb E  \int_0^{t \wedge \tau^R} e^{-\rho s} l(y(s),u(s)) ds\\
  &\quad -\mathbb E  \int_0^{t \wedge \tau^R} e^{-\rho s}  \Big [H_{\rm cv}(y(s),D v(y(s))  ,D^2 v(y(s))  ,u(s)) -H\left (y(s),Dv(y(s), D^2 v(y(s)) \right)   \Big ]   ds.
 \end{aligned}
   \end{equation}
 \item Let  \eqref{eq:behaviour_Ev(y)} and \eqref{eq:HJB(y(t))=0}  hold for every admissible pair $(y(\cdot),u(\cdot))$. Let  $(y^*(\cdot),u^*(\cdot))$ be an admissible pair such that 
 \begin{small}
\begin{equation}\label{eq:sufficient_condition_optimality_argmax}
 u^*(s)\in  {\rm argmin}_{u \in U}   H_{\rm cv}(y^*(s),D v(y^*(s))  ,D^2 v(y^*(s))  ,u), \quad   \mathbb P\textit{-a.s., for  a.e. } s \geq 0.
 \end{equation}
    \end{small}
Then  $(y^*(\cdot),u^*(\cdot))$ is optimal and $v(x)=V(x)$.
 \end{enumerate}
 \end{theorem} 
 \begin{proof}
Let $x \in \mathbb R^n$ and let $(y(\cdot),u(\cdot))$ be an admissible pair for $x$ and let $R>0$.
 As $v \in W^{2,n}_{\rm loc}(\mathbb R^n)$, we can apply Dynkin's formula  for $W^{2,n}_{\rm loc}$-functions (Theorem \ref{th:dynkin_formula}) to obtain
 \begin{equation} \label{eq:Ito_proof}
 \begin{aligned}
v( x)
& =\mathbb E \left[ e^{-\rho (t\wedge \tau^R)}  v(y(t \wedge \tau^R )) \right] +\mathbb E  \int_0^{t \wedge \tau^R} e^{-\rho s}  \Big [ \rho  v(y(s))   -  b(y(s),u(s)) \cdot D  v(y(s))       \\
& \quad 
- \frac 1 2  \tr \left ( \sigma(y(s),u(s))\sigma(y(s),u(s))^T D^2 v(y(s)) \right ) \Big ]   ds\\
& = \mathbb E \left[ e^{-\rho (t\wedge \tau^R)}  v(y(t \wedge \tau^R )) \right] +\mathbb E  \int_0^{t \wedge \tau^R} e^{-\rho s} l(y(s),u(s)) ds\\
 &\quad +\mathbb E  \int_0^{t \wedge \tau^R} e^{-\rho s}  \Big [ \rho  v(y(s))    - H\left (y(s),Dv(y(s)), D^2 v(y(s))  \right)  \Big ]   ds \\
  &\quad +\mathbb E  \int_0^{t \wedge \tau^R} e^{-\rho s}  \Big [H\left (y(s),Dv(y(s), D^2 v(y(s)) \right)-H_{\rm cv}(y(s),D v(y(s))  ,D^2 v(y(s))  ,u(s))    \Big ]   ds,
\end{aligned}
 \end{equation}
where in the second equality we added and subtracted inside the parenthesis the quantities $l(y(s),u(s))$ and $H\left (y(s),Dv(y(s), D^2 v(y(s))  \right)$ and we have used the definition of $H_{\rm cv}$. Next, we proceed as follows:
  \begin{enumerate}[(i)]
  \item for the first point, using \eqref{eq:HJB(y(t))leq0} and the fact that  $H\left (x,p, Z \right)-H_{\rm cv}\left (x,p, Z ,u \right) \leq 0$, we have
 \begin{align*}
v( x) \leq \mathbb E \left[ e^{-\rho (t\wedge \tau^R)}  v(y(t \wedge \tau^R )) \right] +\mathbb E  \int_0^{t \wedge \tau^R} e^{-\rho s} l(y(s),u(s)) ds.
 \end{align*}
Sending $R \to \infty$, then $t \to \infty$, by \eqref{eq:behaviour_Ev(y)} and the dominated convergence theorem, we  obtain
  \begin{align*}
v( x) \leq J(x,u(\cdot))
\end{align*}
and the claim follows by taking $\inf_{u \in \mathcal U}.$
\item For the second point, using \eqref{eq:HJB(y(t))geq0}, we  obtain the claim.
  \item For the third point, we have  $H\left (y^*(s),Dv(y^*(s)), D^2 v(y^*(s))  \right) =
  H_{\rm cv}\left (y^*(s),Dv(y^*(s)), D^2 v(y^*(s)),u^*(s)  \right) $. Therefore, considering \eqref{eq:v_geq_fundamental_rel} with $u(\cdot)=u^*(\cdot)$, letting  $R \to \infty$, $t \to \infty$, and using \eqref{eq:behaviour_Ev(y)}, we have $v(x) \geq J(x,u^*(\cdot))$. Taking also into account \eqref{eq:v_leq_V}, we have
  $$V(x)\geq v(x) \geq J(x,u^*(\cdot)).$$
This implies the claim.
 \end{enumerate}
 \end{proof}
 \begin{remark}\label{rem:strong_solutions}
 Assume that $v\in W^{2,n}_{\rm loc}(\mathbb R^n)$ is a strong subsolution (resp. supersolution, resp. solution) of \eqref{eq:HJB}. Then, for any  admissible pair $(y(\cdot),u(\cdot))$, we have  \eqref{eq:HJB(y(t))leq0} (resp. \eqref{eq:HJB(y(t))geq0}, resp. \eqref{eq:HJB(y(t))=0}). Indeed, denoting by $N$, the Lebesgue null set where $v$ is not twice-differentiable (and the corresponding HJB inequality does not hold), the claims follow thanks to Remark \ref{rem:y_not_in_N}.  This is a key difference with respect to the classical case where $H$ is continuous, $v \in C^2$, the HJB inequalities are satisfied for all $x \in \mathbb R^n$ ($N =\emptyset$). 
 \end{remark}
As a corollary these results hold when $v \in C(\mathbb R^n)$ is an $L^p$-viscosity solution or a viscosity solution, as in this case $v$ turns out to be automatically in $W^{2,n}_{\rm loc}$  (see Section \ref{sec:HJB}).
 \begin{coroll}\label{cor:verification_viscosity_sufficient}
  Let  Assumptions  \ref{hp:measurability}, \ref{hp:locally_bounded_coeff}, \ref{hp:sigma_uniformly_continuous}, \ref{hp:uniform_ellipticity}, \ref{hp:state-eq}, \ref{hp:V}, \ref{hp:H_measurable} hold.  Let $v \in C(\mathbb R^n)$. Fix $x \in \mathbb R^n$ and let \eqref{eq:behaviour_Ev(y)} be satisfied for every admissible pair $(y(\cdot),u(\cdot))$. 
 \begin{enumerate}[(a)]
 \item Let $p \geq n$ and assume that $v $ is an $L^p$-viscosity solution of \eqref{eq:HJB}. Then the conclusions of Theorem \ref{th:verification} (i), (ii), (iii) hold.
 \item  In addition, let Assumption \ref{hp:H_continuous} hold. Assume that $v $ is  a viscosity solution of \eqref{eq:HJB}.  Then the conclusions of Theorem \ref{th:verification} (i), (ii), (iii) hold.
 \end{enumerate}
 \end{coroll}
 \begin{proof}
 \begin{enumerate}[(a)]
 \item By Proposition \ref{th:existence_uniqueness_Lp_viscosity}, we have  $v\in W^{2,p}_{\rm loc}(\mathbb R^n)$  and it is a strong solution to \eqref{eq:HJB}. Then, taking into account Remark \ref{rem:strong_solutions}, the claim follows by Theorem \ref{th:verification}.
 \item  By Proposition \ref{th:existence_uniqueness_viscosity_infinite}, we can proceed as for point (a) to have the claim of (b).
 \end{enumerate}
 \end{proof}
\begin{remark}
  Theorem \ref{th:verification} (together with Remark \ref{rem:strong_solutions}) and Corollary \ref{cor:verification_viscosity_sufficient} can be used to obtain uniqueness of $W^{2,n}$-strong solutions and $L^p$-viscosity solutions of the HJB equation \ref{eq:HJB}, respectively, {see Corollary \ref{coroll:existence_uniqueness}.} 
 \end{remark}

When $v$ is the value function $V$ we have the following further results.
 \begin{theorem}[Verification, sufficient conditions for optimality II]\label{th:verification_II}  Let  Assumptions  \ref{hp:measurability}, \ref{hp:locally_bounded_coeff},  \ref{hp:uniform_ellipticity}, \ref{hp:state-eq}, \ref{hp:V}, \ref{hp:H_measurable} hold. Assume that  $V\in W^{2,n}_{\rm loc}(\mathbb R^n)$.  Fix $x \in \mathbb R^n$ and let $(y^*(\cdot),u^*(\cdot))$ be an admissible pair.  Let    \eqref{eq:behaviour_Ev(y)} and \eqref{eq:HJB(y(t))geq0} hold for $v=V$ and $(y(\cdot),u(\cdot))= (y^*(\cdot),u^*(\cdot))$;  assume \eqref{eq:sufficient_condition_optimality_argmax}.
 Then $(y^*,u^*(\cdot))$ is optimal. 
 \end{theorem}
 \begin{proof}
Consider \eqref{eq:v_geq_fundamental_rel} for $v=V$ and $(y(\cdot),u(\cdot))= (y^*(\cdot),u^*(\cdot))$, use  \eqref{eq:sufficient_condition_optimality_argmax}, and let $R\to \infty, t \to \infty$, then by  \eqref{eq:behaviour_Ev(y)} we have  $V(x)\geq J(x,u^*(\cdot)),$ from which the claim follows.
 \end{proof}
 \begin{theorem}[Verification, necessary conditions for optimality]\label{th:verification_necessary}
Let  Assumptions  \ref{hp:measurability}, \ref{hp:locally_bounded_coeff}, \ref{hp:uniform_ellipticity}, \ref{hp:state-eq}, \ref{hp:V}, \ref{hp:H_measurable} hold. Assume that  $V\in W^{2,n}_{\rm loc}(\mathbb R^n)$. Fix $x \in \mathbb R^n$ and let  $(y^*(\cdot),u^*(\cdot))$  be an optimal pair. Assume that \eqref{eq:behaviour_Ev(y)} and \eqref{eq:HJB(y(t))=0} are satisfied by $v=V$, $(y(\cdot),u(\cdot))=(y^*(\cdot),u^*(\cdot))$.
     Then,  we must have
 \begin{equation}\label{eq:necessary_condition_optimality}
 u^*(s)\in  {\rm argmin}_{u \in U}    H_{\rm cv}(y^*(s),D V(y^*(s))  ,D^2 V(y^*(s))  ,u),  \quad   \mathbb P\textit{-a.s., for  a.e. } s \geq 0.
 \end{equation}
 \end{theorem}
 \begin{proof}
 Since $(y^*(\cdot),u^*(\cdot))$ is an optimal pair, we must have $$V( x)=J(x,u^*(\cdot)).$$ 
 On the other hand, consider  \eqref{eq:Ito_proof}  for $v=V$ and $(y(\cdot),u(\cdot))=(y^*(\cdot),u^*(\cdot))$. Using  \eqref{eq:HJB(y(t))=0} and letting $R \to \infty$, then $t \to \infty$ there, using \eqref{eq:behaviour_Ev(y)}, we have 
  \begin{equation*}
 \begin{aligned}
V( x) & = J(x,u^*(\cdot)) +\mathbb E  \int_0^{\infty } e^{-\rho s}  \Big [ H\left (y^*(s),DV(y^*(s), D^2 V(y^*(s))  \right) -H_{\rm cv}(y^*(s),D V(y^*(s))  ,D^2 V(y^*(s))  ,u^*(s))  \Big ]   ds.
\end{aligned}
 \end{equation*}
Comparing the two identities, we  have
  \begin{equation*}
 \begin{aligned}
0=\mathbb E  \int_0^{\infty } e^{-\rho s}  \Big [ H\left (y^*(s),DV(y^*(s), D^2 V(y^*(s))  \right)-H_{\rm cv}(y^*(s),D V(y^*(s))  ,D^2 V(y^*(s))  ,u^*(s))  \Big ]   ds
\end{aligned}
 \end{equation*}
Since  $0 \geq H\left (x,p, Z \right)-H_{\rm cv} \left (x,p, Z,u \right) ,$ for all $ x,p,Z,u$,  we
conclude that  \eqref{eq:necessary_condition_optimality} must hold.
 \end{proof}
 \begin{coroll}\label{cor:verification_theorem_necessary_corollary} Let  Assumptions \ref{hp:measurability}, \ref{hp:locally_bounded_coeff}, \ref{hp:sigma_uniformly_continuous}, \ref{hp:uniform_ellipticity}, \ref{hp:state-eq}, \ref{hp:V}, \ref{hp:H_measurable} hold.  Fix $x \in \mathbb R^n$ and let  $(y^*(\cdot),u^*(\cdot))$  be an optimal pair. Assume that  \eqref{eq:behaviour_Ev(y)} and \eqref{eq:HJB(y(t))=0} are satisfied by $v=V$, $(y(\cdot),u(\cdot))=(y^*(\cdot),u^*(\cdot))$. 
 Then the following statements hold:
 \begin{enumerate}[(a)]
 \item let $p \geq n$ and assume that $V $ is an $L^p$-viscosity solution of \eqref{eq:HJB}. Then the conclusion of Theorem \ref{th:verification_necessary} holds.
 \item  In addition, let Assumption \ref{hp:H_continuous} hold. Assume that $V$ is  a viscosity solution of \eqref{eq:HJB}.  Then the conclusion of Theorem \ref{th:verification_necessary} holds.
 \end{enumerate}
 \end{coroll}
 \begin{proof}For both points, we proceed as in the proof of Corollary \ref{cor:verification_viscosity_sufficient} and then apply Theorem \ref{th:verification_necessary}. 
 \end{proof}

\section{Optimal feedback controls and characterization of $V$ as the unique $L^p$-viscosity solution}\label{sec:synthesis}
In this section, we use the verification theorems of the previous section to solve the stochastic optimal control problem by constructing  optimal feedback controls.  
 
We assume the following.
\begin{hypothesis}\label{hp.compactness_U}
The infimum in the definition of the Hamiltonian is attained, i.e. 
\begin{align*}
H(x,p,Z) &:=  \min_{u\in U}  H_{\rm cv}(x,p,Z,u),   \quad  \forall x,p \in \mathbb R^n, Z \in S^{n}.
\end{align*}
\end{hypothesis}
\begin{hypothesis}\label{hp:v}
Let   $v\in W^{2,n}_{\rm loc}(\mathbb R^n)$ satisfying either one of the following conditions:
\begin{enumerate}[(a)]
\item $v$ is a strong solution of \eqref{eq:HJB} and   \eqref{eq:behaviour_Ev(y)}  holds for every initial state $x\in\mathbb R^n$ and for every admissible pair $(y(\cdot),u(\cdot))$ starting from $x$. 
\item  $v=V$ and it is a strong supersolution of \eqref{eq:HJB}.
\end{enumerate}
\end{hypothesis}
\begin{flushleft}
    Recall, of course, Propositions \ref{th:existence_uniqueness_Lp_viscosity}, \ref{prop:existence_viscosity}, \ref{th:existence_uniqueness_viscosity_infinite} for sufficient conditions granting the existence of strong solution of \eqref{eq:HJB} in the case of $L^p$-viscosity solutions or viscosity solutions.
\end{flushleft}

Define the set-valued map 
  \begin{align}\label{eq:set_valued_map}
  \Psi \colon \mathbb R^n \to  \mathcal  P (U), \quad  \Psi(x)&: ={\rm argmin}_{u \in U}   H_{\rm cv}(x,Dv(x),D^2v(x),u)  \neq \emptyset, \quad  x \in \mathbb R^n,
  \end{align}
  which takes non-empty set values due to  Assumption \ref{hp.compactness_U}. A selection of the set-valued map $\Psi$ is a map $\psi \colon \mathbb R^n \to  U$  such that $\psi(x)\in \Psi(x)$ for {a.e.} $x \in \mathbb R^n$.
   \begin{prop}[Optimal feedback controls I]\label{cor:closed_loop}
Let  Assumptions \ref{hp:measurability}, \ref{hp:locally_bounded_coeff},  \ref{hp:uniform_ellipticity},  \ref{hp:V}, \ref{hp.compactness_U}, \ref{hp:v} hold  and let $x \in \mathbb R^n$.
 Assume that $\Psi$ has a Borel  measurable selection $\psi \colon \mathbb R^n \to  U$ such that the closed loop equation
  \begin{equation}
\label{eq:closed_loop}
dy(t) = b(y(t),\psi(y(t))  dt + \sigma(y(t),\psi(y(t))) \,dW(t), \quad y(0) = x,
\end{equation}
admits a solution $y(t)$ with continuous paths a.s.  in some generalized reference probability space $\eta$, such that 
\begin{equation}\label{eq:integrability_y_psi_y}
 \mathbb E\left[\int_0^{\infty} e^{-\rho t}\left |
l(y(t),\psi(y (t))) \right |dt \right] < \infty.
\end{equation}
Set $u(t) :=  \psi(y (t))$. In case (b) above, we also assume that   \eqref{eq:behaviour_Ev(y)}  holds for $(y(\cdot),u(\cdot))$. Then, in either cases (a), (b), the pair $(y(\cdot),u(\cdot))$ is optimal and $v(x)=V(x)$.
 \end{prop}
\begin{proof}
Note that, as $y(t)$ is $\mathcal F_t$-progressively measurable and $\psi$ is Borel measurable, we have that $u(\cdot)$ is progressively measurable. It follows that $y(t)$ is a solution to \eqref{eq:SDDE}  with $x, u(\cdot)$ in the generalized reference probability space $\eta$ and the pair $(y(\cdot),u(\cdot))$ is admissible for $x$ (hence, Assumption \ref{hp:state-eq} is satisfied at $x$).  By construction, $u(\cdot)$ satisfies
  $$u(s)\in  {\rm argmin}_{u \in U}    H_{\rm cv}(y(s),D v(y(s))  ,D^2 v(y(s))  ,u), \quad   \mathbb P\textit{-a.s., for  a.e. } s \geq 0.$$
 Hence, by  Theorem \ref{th:verification} for case (a) and by Theorem \ref{th:verification_II} for case (b), together with Remark \ref{rem:strong_solutions}, we conclude that the pair $(y(\cdot),u(\cdot))$ is optimal and $v(x)=V(x)$. 
\end{proof}

In the following proposition,  under boundedness of $b,\sigma$,  we  solve the closed loop equation and construct optimal feedback laws. 
\begin{prop}[Optimal feedback controls II]\label{prop:optimal_feedback_example}
 Let  Assumptions \ref{hp:measurability}, \ref{hp:b,sigma,bounded}, \ref{hp:l_subpolynomial_growth}, \ref{hp:uniform_ellipticity_strong},  \ref{hp.compactness_U}, \ref{hp:v} hold. Then, there exists a Borel measurable selection $\psi \colon \mathbb R^n \to  U$ of $\Psi$, independent of the initial state, such that for every $x\in\mathbb R^n$, the closed loop equation \eqref{eq:closed_loop} has a solution $y(t)$ in some generalized reference probability space $\eta$. Moreover, for each $x$, in either cases (a), (b), the pair $(y(\cdot),u(\cdot))$, where $u(t):=\psi(y(t))$, is admissible and optimal for $x$ and $v(x)=V(x)$.
 \end{prop}
 \begin{proof}
{By Remark \ref{rem:borel_meas_implies_lsa}, we can apply Theorem \ref{th:meas_selection_univ}, to have that $\Psi$  admits a universally measurable selection $\psi \colon \mathbb R^n \to  U$. Then,  we can find a Borel measurable map $\tilde \psi \colon \mathbb R^n \to  U$ such that $\tilde \psi= \psi$, a.e. on $\mathbb R^n$ (this can be seen, e.g., by writing $\mathbb R^n=\bigcup_{k\in \mathbb N}C_k$, where $\{C_k\}_k\subset \mathbb R^n$ is a disjoint sequence of hypercubes of hypervolume $1$ in the Lebesgue measure covering $\mathbb R^n$, so that,  by \cite[Exercise 1, p. 265]{cohn},  we can find Borel measurable functions $\tilde \psi_k:\mathbb R^n\to U$, such that $\tilde \psi_k= \psi$ on $C_k$, a.e., for all $k$;  finally we define $\tilde \psi(x)=\tilde \psi_k(x)$,  $x\in C_k,k\in \mathbb N$).} In turn, the functions 
$$\tilde b \colon \mathbb R^n \to \mathbb R^n, \quad \tilde b(x):=b(x,{\tilde \psi}(x)), \quad \tilde \sigma \colon \mathbb R^n \to \mathbb R^{n\times n}, \quad  \tilde \sigma(x):=\sigma(x,{\tilde \psi}(x))$$ are Borel  measurable and bounded. We fix this Borel selection independently of the initial state and henceforth write $\psi$ for $\tilde\psi$. Hence, thanks to Assumption \ref{hp:uniform_ellipticity_strong}, we can apply \cite[Theorem 1, p.87]{krylov} to obtain, for every initial state $x\in\mathbb R^n$, a continuous solution $y(t)$ to the closed loop equation \eqref{eq:closed_loop} in some generalized reference probability space $\eta$. 

Next, we show that Hypothesis \ref{hp:V} and 
\eqref{eq:integrability_y_psi_y} hold. Indeed, due to the boundedness of $b, \sigma,$ for any solution $y(t)$ of the state equation \eqref{eq:SDDE} with control process $u(t)$ and initial datum $x$, for all $\bar m\geq 1$,  we have
\begin{equation}\label{eq:state_est}
    \mathbb E[\sup_{0\leq s\leq t}|y(s)|^{\bar m}] \leq  C {(1+t^{\bar m}+t^{\bar m/2}) } (1+|x|^{\bar m})
\end{equation}
 ($C>0$ independent of $u(\cdot),x,t$), so that, thanks to Assumption \ref{hp:l_subpolynomial_growth},
\begin{align*}
\mathbb E\left[\int_0^{\infty} e^{-\rho t}\left |
l(y(t),u(t)) \right | dt \right]  \leq C \mathbb \int_0^{\infty} e^{-\rho t}
(1+E\left[|y(t)|^m\right]) dt\leq C  (1+|x|^m) \mathbb \int_0^{\infty} e^{-\rho t} {(1+t^{m}+t^{ m/2}) } dt   \leq C (1+|x|^m).
\end{align*}
Together with the closed-loop solutions constructed above, the cost bound proves that $\mathcal U\neq\emptyset$ for every initial state $x$. The uniform cost bound gives $|V(x)|\leq C(1+|x|^m)$. Moreover, \eqref{eq:state_est} implies that every continuous function $w$ with polynomial growth satisfies $\mathbb E[\sup_{0\leq s\leq t}|w(y(s))|]<\infty$ for every $t>0$ and $e^{-\rho t}\mathbb E[|w(y(t))|]\to0$ as $t\to\infty$. Hence \eqref{eq:behaviour_Ev(y)} holds with $v=w$, in particular with $v=V$ in case (b).
This implies that Hypothesis \ref{hp:V} is satisfied; taking $u(t):=\psi(y(t))$, where $y(t)$ is the solution to the closed loop equation \eqref{eq:closed_loop}, then \eqref{eq:integrability_y_psi_y} also holds. Therefore, the statement of the proposition follows by Proposition \ref{cor:closed_loop}.
 \end{proof}
  
{\begin{coroll}[Existence and uniqueness of $L^p$-viscosity solutions]\label{coroll:existence_uniqueness}Let  Assumptions \ref{hp:measurability}, \ref{hp:b,sigma,bounded}, \ref{hp:l,bounded},  \ref{hp:sigma_uniformly_continuous}, \ref{hp:uniform_ellipticity_strong},  \ref{hp:H_measurable}, \ref{hp.compactness_U}. Let  $p \geq n$. Then $V$ is the unique $L^p$-viscosity solution to \eqref{eq:HJB} in  the class $C(\mathbb R^n)$ with polynomial growth. Moreover, $V$ is the unique strong solution to  \eqref{eq:HJB} in the class of  functions in $W^{2,p}_{\rm loc}(\mathbb R^n)$ with polynomial growth.
 \end{coroll}}
  \begin{proof}
{Existence: Proposition \ref{prop:existence_viscosity}
gives existence of an $L^p$-viscosity solution $v \in C_b(\mathbb R^n)$ to \eqref{eq:HJB}. In turn, by Proposition \ref{th:existence_uniqueness_Lp_viscosity} we have $v \in W^{2,n}_{\rm loc}(\mathbb R^n)$ and it is also a strong solution.}

{Uniqueness:  let   $v \in W^{2,p}_{\rm loc}(\mathbb R^n)$ be a strong solution (or, equivalently, by Proposition \ref{th:existence_uniqueness_Lp_viscosity}, let $v \in C(\mathbb R^n)$ be an $L^p$-viscosity solution)  to \eqref{eq:HJB}. Using \eqref{eq:state_est} and the polynomial growth of  $v$, we have that $v$ satisfies Assumption \ref{hp:v} (a). Then,
 Proposition \ref{prop:optimal_feedback_example} implies $v=V$.}
   \end{proof}
 
\section{Application to stochastic optimal advertising models}\label{sec:application}
In this section, we solve a  stochastic optimal control  problem arising in economics within the context of optimal advertising (see e.g. \cite{grosset_viscolani,nerlove_arrow,deFeo_2023,defeo_federico_swiech,gozzi_marinelli_2004} and the references therein).

 Consider a firm who seeks to optimize the advertising campaign for a certain product.
Assume that the dynamics of the stock of the advertising goodwill $y(t)$ of the product is given by the following controlled 1-dimensional SDE
 \begin{equation*} 
 dy(t) = \left[ a (y(t))+c u(t) \right]          dt   + [\nu(y(t)) +\gamma u(t)] dW(t),\quad   y(0)=x,  
 \end{equation*}   
 where $W(t)$ is a real-valued Brownian motion, representing the uncertainty in the market; the control process $u(t)$ with values in $U:= [0,\bar u]$,   $\bar u>0$, models the investment rate;  $a \colon \mathbb R \to (-\infty,0]$  is a bounded, Borel measurable function (hence, it is allowed to be discontinuous),  representing image deterioration under different regimes, depending on the level of the goodwill;   $c  > 0$ is an investment effectiveness factor;   $\nu\colon \mathbb R \to (0,\infty)$ and $ \gamma \geq 0$  represent the intensity of the uncertainty in the model. We assume that $\nu$ is a bounded uniformly continuous function such that $\nu(x)\geq \delta$ for $\delta>0$.
 
 The goal is to minimize, over all admissible controls $u(\cdot)$,
 $$ \mathbb{E} \left[\int_0^\infty e^{-\rho s} (h(u(s))-g(y(s))) d s\right], $$ where  $\rho >0$ is a discount factor,  $h \colon U \rightarrow \mathbb R$  is a continuous (bounded) cost function  and $g \colon \mathbb R \rightarrow \mathbb R$ is a bounded Borel measurable  utility function.

Setting $b(x,u):=a (x)+c u, \sigma(x,u):=\nu(x)+\gamma u, l(x,u):=h(u)-g(x)$, we are in the setting of Section \ref{sec:formul}. Then, with the notations of Section \ref{sec:HJB}, we have a {Borel} measurable Hamiltonian
\begin{small}
\begin{align*}
 &H(x,p,Z) =a(x)p - g(x)+ \inf_{u\in U}\left[ cu p  + {1\over 2} (\nu(x)+\gamma u)^2  Z+h(u) \right], &    x,p,Z \in \mathbb R, u \in U.
\end{align*}
 \end{small}
{For $p \geq n=1$ we can apply Proposition \ref{prop:optimal_feedback_example} and Corollary \ref{coroll:existence_uniqueness} (notice that $H$ satisfies Assumption \ref{hp.compactness_U} (due to the compactness of $U$ and the continuity of all the functions of $u$) to completely solve the optimal advertising problem by 
 \begin{itemize}
     \item characterizing the value function $V$  as the unique  $L^p$-viscosity solution in $C_b(\mathbb R)$  to the HJB equation \eqref{prop:existence_viscosity}, which here reads
\begin{equation}
\rho v -a(x)Dv + g(x)-  \inf_{u\in U}\left[ cu Dv  +  {1\over 2} (\nu(x)+\gamma u)^2  D^2v+ h(u) \right] =0,
\quad \forall x \in \mathbb R,
\end{equation}
\item finding a weak solution $y(t)$ of the closed loop equation \eqref{eq:closed_loop}, where 
 $\psi$ is  a Borel measurable selection of 
 the set valued map
  \begin{align*}
 \Psi(x) ={\rm argmin}_{u \in U} \left[ cu DV(x)  + {1\over 2} (\nu(x)+\gamma u)^2  D^2V(x)+ h(u)\right] , \quad  x \in \mathbb R.
  \end{align*}
  Therefore, the optimal policy is $u(t):=\psi(y(t))$.
 \end{itemize}}
 
 \appendix
 \section{Notations}\label{sec:notations}
 \subsection{Basic notations}Throughout the paper, we  work with Von Neumann–Bernays–Gödel set theory (NBG) \cite{mendelson}\footnote{Recall that NBG is a conservative extension of Zermelo-Fraenkel set theory (with the Axiom of Choice)}, as we work with proper classes (i.e. classes which are not sets).
 We denote by  $\mathbb R^{m \times n}$ the space of real valued $m \times n$-matrices and we denote by $|\cdot|$ the Euclidean norm in $\mathbb{R}^{n}$ as well as the Frobenius norm in $\mathbb R^{m \times n}$. We denote by $x \cdot y$ the inner product in $\mathbb R^n$. For $R>0$ we denote $B_R:=\{x \in \mathbb R^n: |x|< R \}.$ 
 We denote by $S^n$ the space of $n\times n$-symmetric matrices. We will use the order relation $\geq$ on  $S^n$, defined by $Y \geq Z$ if $(Yx) \cdot x \geq (Zx)\cdot x$ for every $Y,Z \in S^n, x \in \mathbb R^n$.

 {Given a topological space $X$, we denote by $\mathcal B(X)$ the Borel $\sigma$-algebra.}
 
  We write $C>0,\omega, \omega_R$ to indicate, respectively, a constant, a modulus of continuity, and a local modulus of continuity, which may change from line to line if the precise dependence on other data is not important. 
 
 Let $O \subset \mathbb R^n$ be open.
If $v \colon O \to \mathbb R$ is differentiable  at $x \in O$, we denote by $D v(x)$ its gradient at $x$ and if it is twice differentiable at $x \in O$, we denote by $D^2 v(x)$ its Hessian matrix at $x$.  We will denote by $C^k(\mathbb R^n),C^k_b(\mathbb R^n)$, $k \in \mathbb N$, respectively, the space of $k$-times continuously differentiable functions, $k$-times continuously differentiable functions with bounded $k$-th  derivative.  For $0<\alpha \leq 1$, we will denote by $C^{1,\alpha}$ the space of  continuously differentiable functions with $\alpha$-Holder's continuous derivatives. 

Throughout the whole paper, given a Lebesgue measurable function $v \colon O \to \mathbb R$, we will pick   a representative member of its equivalence class with respect to the a.e. relation (with the Lebesgue measure), still denoted by $v$, which is Borel measurable.
For $p \geq 1$, we will denote by $L^p(O;\mathbb R^{n \times m})$  the space of equivalence classes with respect to the a.e. relation  of $p$-integrable functions $v \colon O \to \mathbb R$. 

We will denote by $W^{1,p}(O)$, the Sobolev space of functions $v \in L^p(O)$ such that its distributional gradient  $Dv  \in L^p(O,\mathbb R^n)$ and by $W^{2,p}(O)$, the Sobolev space of functions $v \in L^p(O)$ such that its distributional gradient and Hessian, respectively, $Dv  \in L^p(O,\mathbb R^n)$, $D^2v  \in L^p(O,\mathbb R^{n\times n})$. As said above, we will always pick representative members the equivalence classes of $v, Dv, D^2v$, still denoted by $v, Dv, D^2v$,  which are Borel measurable.

 The {subscript} $_{\rm loc}$ will indicate that any of the previous properties is valid on every compact set contained in $O.$

 \subsection{Properties of spaces $W^{2,p}$}\label{subsec:W2p}
 We will recall here  standard properties of functions in $ W^{2,p}_{\rm loc}(O)$.
 
Let   $O \subset \mathbb R^n$ be open with smooth boundary and consider the space $W^{2,p}_{\rm loc}(O)$ for $p\geq 1$. Hence, $v \in W^{2,p}_{\rm loc}(O)$ is such that $Dv  \in L^p_{\rm loc}(O,\mathbb R^n)$, $D^2v  \in L^p_{\rm loc}(O,\mathbb R^{n\times n})$.  
In this case, we have:
\begin{enumerate}[(i)]
\item if $p>n/2$, by Sobolev embeddings we have that (up to a representative of the equivalence class) $v \in C^{0,\alpha}_{\rm loc}(O)$  for some $\alpha\in(0,1)$, so that $v$  is continuous on $O$. Moreover, $v$ is twice differentiable a.e. on $O$ and its a.e.-gradient (resp. a.e.-Hessian) is equal to  its distributional gradient $Dv \in L^{p}_{\rm loc}(O;\mathbb R^n)$ (resp. Hessian $D^2v\in L^{p}_{\rm loc}(O;\mathbb R^{n \times n})$) (see \cite[Proposition 2.2, Appendix C]{caffarelli_c_s});
\item If $p>n,$ we further have $v \in C^{1,\alpha}_{\rm loc}(O)$ for some $\alpha\in(0,1)$ (by Sobolev embeddings), so that $v$ is differentiable on $O$ and its pointiwse gradient coincides with its distributional gradient $Dv$. 
\end{enumerate}

  \section{Viscosity solutions and $L^p$-viscosity solutions}\label{sec:appendix_definitions}
Consider a second order  fully non-linear partial differential equation of the form
\begin{equation}\label{eq:appendix_Lp_viscosity_sol}
F\left(x, u, D u, D^2 u\right)=f(x) \quad \forall x \in  \Omega,
\end{equation}
where  $\Omega \subset \mathbb{R}^n$ is a domain (not necessarily bounded), $F: \Omega \times \mathbb{R} \times \mathbb{R}^n \times S(\mathbb R^n) \rightarrow \mathbb{R}$ is measurable and such that $F(x,0,0,0)=0$ and $f \in L^p_{\rm loc}(\Omega)$, $p >n/2$.  We use the following definitions. 
\begin{definition}[Strong solution]
A function $u \in W^{2,p}_{\rm loc}(\Omega)$ is called
\begin{enumerate}[(i)]
\item a strong subsolution  of \eqref{eq:appendix_Lp_viscosity_sol} if
$
F\left(x, u(x), Du(x), D^2u(x)\right) \leq f(x)
$ for a.e. $x \in \Omega$;
\item a strong supersolution  of \eqref{eq:appendix_Lp_viscosity_sol} if
$
F\left(x, u(x), Du(x), D^2u(x)\right) \geq f(x)
$ for a.e. $x \in \Omega$;
\item  a strong solution  of \eqref{eq:appendix_Lp_viscosity_sol} if $
F\left(x, u(x), Du(x), D^2u(x)\right) = f(x)
$ for a.e. $x \in \Omega$.
\end{enumerate}
\end{definition}
\begin{definition}[Viscosity solution]\label{def:viscosity_solution_finite_dim}
Let $F,f$ be continuous. A function $u \in C( \Omega)$ is called
\begin{enumerate}[(i)]
\item a viscosity subsolution  of \eqref{eq:appendix_Lp_viscosity_sol} if
$
F\left(x, u(x), D \varphi(x), D^2 \varphi(x)\right)\leq f(x) ,
$
whenever, for $\varphi \in C^2(\Omega),$ the function  $u-\varphi$ attains a local maximum at $x \in \Omega$.
\item a viscosity supersolution of \eqref{eq:appendix_Lp_viscosity_sol} if
$
F\left(x, u(x), D \varphi(x), D^2 \varphi(x)\right) \geq f(x),
$
whenever, for $\varphi \in C^2(\Omega),$ the function $u-\varphi$ attains a local  minimum at $x \in \Omega$.
\item  a viscosity solution of \eqref{eq:appendix_Lp_viscosity_sol} if it is both a viscosity subsolution and a viscosity supersolution of \eqref{eq:appendix_Lp_viscosity_sol}.
\end{enumerate}
\end{definition}
When $F$ is not continuous, the notion typically used is the one of the so called $L^p$-viscosity solutions. In this case the set of test functions is enlarged to functions  in $W_{\text {loc }}^{2, p}(\Omega)$.
\begin{definition}[$L^p$-viscosity solution]\label{def:Lp_viscosity}
A function $u \in C(\Omega)$ is called
\begin{enumerate}[(i)]
\item an $L^p$-viscosity subsolution  of \eqref{eq:appendix_Lp_viscosity_sol} if
$
\operatorname{e s s} \liminf _{y \rightarrow x}F\left(y, u(y), D \varphi(y), D^2 \varphi(y)\right) - f(y)\leq 0 ,
$
whenever, for $\varphi \in W_{\text {loc }}^{2, p}(\Omega),$ the function  $u-\varphi$ attains a local maximum at $x \in \Omega$.
\item an $L^p$-viscosity supersolution of \eqref{eq:appendix_Lp_viscosity_sol} if
$
\operatorname{e s s}  \limsup _{y \rightarrow x}F\left(y, u(y), D \varphi(y), D^2 \varphi(y)\right)- f(y) \geq 0,
$
whenever, for $\varphi \in W_{\text {loc }}^{2, p}(\Omega),$ the function $u-\varphi$ attains a local  minimum at $x \in \Omega$.
\item  an $L^p$-viscosity solution of \eqref{eq:appendix_Lp_viscosity_sol} if it is both an $L^p$-viscosity subsolution and an $L^p$-viscosity supersolution of \eqref{eq:appendix_Lp_viscosity_sol}.
\end{enumerate}
\end{definition}
\begin{prop}\label{prop:equivalence_viscosity_Lpviscosity}
Let $p \geq n$;  assume that $F,f$ are continuous and  \cite[Structure condition (SC)]{caffarelli_c_s} is satisfied. Then, viscosity (sub-, super-) solutions of \eqref{eq:appendix_Lp_viscosity_sol}  and $L^p$-viscosity (sub-, super-) solutions of \eqref{eq:appendix_Lp_viscosity_sol} are equivalent (see \cite{caffarelli_c_s}).
\end{prop}

 \section{$W^{2,n}$-Dynkin's formula}\label{app:ito}
Let $x \in \mathbb R^n$; for an admissible pair $(y(\cdot),u(\cdot))$ (see Definition \ref{def:weak-form}), recall the definition of the stopping time $\tau^R$, $R>0$ given in \eqref{eq:def_stopping_time}.
 \subsection{Estimates for $W^{2,n}_{\rm loc}$-functions}
The following estimates  \cite[Theorem 2, p.52; Theorem 4, p. 54]{krylov}  grant that Dynkin's formulas for $W^{2,n}$-functions are well-posed.
 \begin{theorem}\label{th:estimate_krylov}
 Let  Assumptions \ref{hp:measurability}, \ref{hp:uniform_ellipticity} be satisfied and assume that there exists $C_R>0$ such that $|b(x,u)| \leq C_R$, for every $x \in \overline B_R, u \in U$.  Then, for every $R>0$, there exists $K_{R,\lambda_R}>0$ such that  for every $f \in L^n_{\rm loc}(\mathbb R^n)$ and every admissible pair $(y(\cdot),u(\cdot))$, it holds
\begin{equation*}
 \mathbb E \left[  \int_0^{t \wedge \tau^R} e^{-\rho s}  \left | f(y(s))  \right|    ds \right]  \leq K_{R,\lambda_R} |f|_{L^n(B_R)}.
 \end{equation*}
 \end{theorem}
 In particular, the previous estimate implies that the process avoids  sets of measure zero $\mathbb P$-a.s.:
\begin{remark}\label{rem:y_not_in_N}
 Let $N \subset \mathbb R^n$ be a Lebesgue measurable subset with zero Lebesgue measure; then  $y(s) \not \in N$ $\mathbb P$-a.s., for a.e. $s \geq 0$.
Indeed, from Theorem \ref{th:estimate_krylov} with  $f=I_N$ we have $I_N(y(s))=0$ for a.e. $s \geq 0$, i.e. the claim.
 \end{remark}
 \subsection{$W^{2,n}_{\rm loc}$-Dynkin's formula.} The  following Dynkin's formula follows from \cite[Theorem 1, p.122]{krylov} (recall the definitions used there and introduced in  \cite[Chapter 2, Section 1]{krylov}. In particular, the Author has the correspondence $W^{2}=W^{2,n}$).  
 \begin{theorem}[$W^{2,n}_{\rm loc}(\mathbb R^n)$-Dynkin's formula]\label{th:dynkin_formula} Let Assumptions \ref{hp:measurability}, \ref{hp:locally_bounded_coeff} (for $b,\sigma$), \ref{hp:uniform_ellipticity},  \ref{hp:state-eq}  hold.
 Let $v \in W^{2,n}_{\rm loc}(\mathbb R^n)$ {and  $(y(\cdot),u(\cdot))$ an admissible pair.} Then
 \begin{equation}\label{eq:Ito_W2}
 \begin{aligned}
\mathbb E \left[ e^{-\rho (t\wedge \tau^R)}  v(y(t \wedge \tau^R )) \right]&=v( x) +\mathbb E \Bigg[  \int_0^{t \wedge \tau^R} e^{-\rho s}  \Big [ -\rho  v(y(s))   +  b(y(s),u(s)) \cdot D  v(y(s))       \\
& \quad 
+ \frac 1 2  \tr \left ( \sigma(y(s),u(s))\sigma(y(s),u(s))^T D^2 v(y(s)) \right ) \Big ]   ds \Bigg ] .\\
\end{aligned}
 \end{equation}
 \end{theorem}
 We recall (see Appendix \ref{sec:notations}) that we are picking  representative members of the corresponding equivalence classes of $v,Dv, D^2v$ which are, respectively, continuous, Borel measurable, Borel measurable. As explained there, such derivatives have both distributional and a.e.-pointwise  sense.

We refer the reader to \cite{chitashvili_mania} for an Ito's formula in a similar setting and to \cite{cheng_deangelis,gozzi_russo_1} for extensive references of other non-smooth Ito's formulas available in the literature.
 
\section{A measurable selection theorem in optimization}\label{app:measurable-selection}
{In this section, we recall a measurable selection theorem in optimization theory from \cite[Appendix A.2]{bersektas}.}

{Throughout this section, let $X,Y$ be Borel spaces. If $h:X\times Y\to [-\infty,\infty]$ is Borel measurable, the partial infimum  and the argmin functions are not necessarily Borel measurable. This issue can be solved by enlarging the class of measurable sets and  measurable functions by considering the following classes. We denote the projection mapping $\operatorname{proj}_X:X\times Y\to X$,  $\operatorname{proj}_X(x,y)=x$.
\begin{definition}[Analytic set]
    A subset $A\subset X$ is said to be analytic if there exists a Borel space $Y$ and a Borel subset $B\subset X \times Y$ such that $A=\operatorname{proj}_X(B)$. 
\end{definition}
It is clear that every Borel subset of a Borel space is analytic.
\begin{definition}[Lower semianalytic function]
We say that $h: X \mapsto[-\infty, \infty]$ is lower semianalytic if the level set
$
\{x \in X : h(x)<c\}
$
is analytic for every $c \in \mathbb R$. 
\end{definition}
\begin{remark}\label{rem:borel_meas_implies_lsa}
    If $h: X \to[-\infty, \infty]$ is Borel measurable then it is lower semianalytic.
\end{remark}
\begin{definition}[Universal $\sigma$-algebra $\mathcal{U}_X$]\label{def:univ_sigma_alg}
    The universal $\sigma$-algebra $\mathcal{U}_X$  is defined as the intersection of all completions of $\mathcal{B}_X$ with respect to all probability measures on $(X,\mathcal B_X)$. Thus, $E \in \mathcal{U}_X$ if and only if, given any probability measure $p$ on $\left(X, \mathcal{B}_X\right)$, there is a Borel set $B$ and a $p$-null set $N$ such that $E=B \cup N$. Clearly, we have $\mathcal{B}_X \subset \mathcal{U}_X$.
\end{definition}
 It is also true that every analytic set is universally measurable, and hence the $\sigma$-algebra generated by the analytic sets, called the analytic $\sigma$-algebra, and denoted $\mathcal{A}_X$, is contained in $\mathcal{U}_X$ :
$\mathcal{B}_X \subset \mathcal{A}_X \subset \mathcal{U}_X$.
\begin{definition}[Universally measurable function]\label{def:univ_meas}
We say that  a function $h: X \to Y$ is universally measurable if $h^{-1}(B) \in \mathcal{U}_X$ for every $B \in \mathcal{B}_Y$. 
\end{definition}
\begin{prop}[Measurable Selection Theorem]\label{th:meas_selection_univ} Let $D \subset X \times Y$ be an analytic set and $h: D \rightarrow[-\infty, \infty]$ a lower semianalytic function. For $x \in X$, let $D_x=\{y \in Y :(x, y) \in D\}$. Define $h^*: \operatorname{proj}_X(D) \rightarrow[-\infty, \infty]$ by $h^*(x)=\min _{y \in D_x} h(x, y)$, where we assume that the minimum is attained at any $x$. Then $h^*$ is lower semianalytic  and there exists a universally measurable  function $\psi: \operatorname{proj}_X(D) \rightarrow Y$ such that $(x, \psi(x)) \in D$ for every $x \in \operatorname{proj}_X(D)$ and
$$
\begin{aligned}
h(x, \psi(x)) & =h^*(x),\quad \forall x \in \operatorname{proj}_X(D).
\end{aligned}
$$
\end{prop}}
\paragraph{\textbf{Acknowledgments.}} The author is grateful to Andrzej Święch {and the Referee} for many helpful comments related to the content of the manuscript and to Fausto Gozzi and Giorgio Fabbri for discussions related to optimal advertising models with discontinuous coefficients.  The author acknowledges funding by the Deutsche Forschungsgemeinschaft (DFG, German Research Foundation) – CRC/TRR 388 ``Rough Analysis, Stochastic Dynamics and Related Fields'' – Project ID 516748464, by INdAM (Instituto Nazionale di Alta Matematica F. Severi) - GNAMPA (Gruppo Nazionale per l'Analisi Matematica, la Probabilità e le loro Applicazioni), and by the Italian Ministry of University and Research (MIUR), in the framework of PRIN projects 2017FKHBA8 001 (The Time-Space Evolution of Economic Activities: Mathematical Models and Empirical Applications) and 20223PNJ8K (Impact of the Human Activities on the Environment and Economic Decision Making in a Heterogeneous Setting: Mathematical Models and Policy Implications).

\bibliographystyle{amsplain}

\begin{thebibliography}{99}




\bibitem{bersektas}
D. P. Bertesekas, \textit{Dynamic programming and optimal control.} Athena scientific 143 (1995).


\bibitem{caffarelli_c_s}
L. Caffarelli, M. G. Crandall, M. Kocan, A. \'{S}wi\k{e}ch, \textit{On viscosity solutions of fully nonlinear equations with measurable ingredients}. Comm. Pure Appl. Math. 49 (1996), no. 4, 365–397.




\bibitem{cheng_deangelis}
C. Cheng, T. De Angelis, \textit{A change of variable formula with applications to multi-dimensional optimal stopping problems}. Stoch. Process. Their Appl. 164 (2023), 33-61.

\bibitem{chitashvili_mania}
R. Chitashvili, M. Mania, \textit{Generalized Itô formula and derivation of Bellman’s equation}. Stochastics Monogr., 10
Gordon and Breach Science Publishers, Yverdon, (1996), 1-21. 

\bibitem{cohn}
D. L. Cohn,
\textit{Measure Theory}. 2nd ed.,
Birkh\"auser/Springer, New York, 2013.




    \bibitem{crandall_kocan_soravia_swiech}
    M. Crandall, M. Kocan, P. Soravia, A. Świech, \textit{On the equivalence of various weak notions of solutions of elliptic PDEs with measurable ingredients.} Progress in Elliptic and Parabolic Partial Differential Equations, Pitman Res. Notes Math. Ser., Longman, Harlow, 350 (1996), 136–162.


    \bibitem{crandall_kocan_swiech}
M. Crandall, M. Kocan, A. Świech, \textit{$L^p$-theory for fully nonlinear uniformly parabolic equations: Parabolic equations.} Commun. Partial Differ. Equ. 25 (2000), no. 11-12, 1997-2053.

    \bibitem{criens}
    D. Criens, \textit{Stochastic optimal control problems with measurable coefficients and $L_d$-drift.} 	arXiv:2404.17236v4 (2025)


\bibitem{defeo_phd}
F. de Feo, \textit{SDEs on Hilbert spaces: slow-fast systems of SPDEs, stochastic optimal control with delays and applications to economics and finance}. PhD Thesis, Politecnico di Milano (2024) https://hdl.handle.net/10589/220474

\bibitem{deFeo_2023}
F. de Feo, \textit{Stochastic optimal control problems with delays in the state and in the control via viscosity solutions and applications to optimal advertising and optimal investment problems}. Decis. Econ. Finance (2024) 31 pp.


\bibitem{defeo_federico_swiech}
F. de Feo, S. Federico, A. \'{S}wi\k{e}ch, \textit{Optimal control of stochastic delay differential equations and applications to path-dependent financial and economic models.} SIAM J. Control Optim. 62 (2024), no. 3, 1490–1520.




\bibitem{deFeoSwiech}
F. de Feo, A. \'{S}wi\k{e}ch, \textit{Optimal control of stochastic delay differential equations: Optimal feedback controls.} J. Differ. Equ. 420 (2025) 450-508.

\bibitem{defeo_swiech_wessels}
F. de Feo, A. \'{S}wi\k{e}ch, L. Wessels, \textit{Stochastic optimal control in Hilbert spaces: $C^{1,1}$ regularity of the value function and optimal synthesis via viscosity solutions.} arXiv preprint, arXiv:2310.03181 (2023).

\bibitem{du_wei}
K. Du, Q. Wei. \textit{Optimal control of SDEs with merely measurable drift: an HJB approach.} arXiv preprint arXiv:2509.01054 (2025).

  \bibitem{fgs_book}
  G. Fabbri, F. Gozzi, A. \'{S}wi\k{e}ch, \textit{Stochastic optimal control in infinite dimension. Dynamic programming and HJB equations. With a contribution by Marco Fuhrman and Gianmario Tessitore.} Probability Theory and Stochastic Modelling, 82. Springer, Cham (2017).

\bibitem{federico_gozzi_2018}
S. Federico, F. Gozzi, \textit{Verification theorems for stochastic optimal control problems in Hilbert spaces by means of a generalized Dynkin formula.} Ann. Appl. Probab. 28 (2018), no. 6, 3558–3599.


\bibitem{fleming_soner}
W. H. Fleming, H. M. Soner, \textit{Controlled Markov processes and viscosity solutions.} Vol. 25. Springer Science $\&$ Business Media, 2006.


\bibitem{gozzi_marinelli_2004}F. Gozzi, C. Marinelli, \textit{Stochastic optimal control of delay equations arising in advertising models.} Stochastic partial differential equations and applications—VII, 133–148, Lect. Notes Pure Appl. Math., 245, Chapman $\&$ Hall/CRC, Boca Raton, FL, 2006.

\bibitem{gozzi_russo}
F. Gozzi, F. Russo, \textit{Verification theorems for stochastic optimal control problems via a time dependent Fukushima–Dirichlet decomposition.} Stoch. Process. Their Appl. 116 (2006), no. 11, 1530-1562.

\bibitem{gozzi_russo_1}
F. Gozzi, F. Russo, \textit{Weak Dirichlet processes with a stochastic control perspective.} Stoch. Process. Their Appl. 116 (2006), no. 11, 1563-1583.





\bibitem{gozzi_swiech_zhou_2005}
F. Gozzi, A. \'{S}wi\k{e}ch, X.Y. Zhou, \textit{A corrected proof of the stochastic verification theorem within the framework of viscosity solutions.} SIAM J. Control Optim. 43 (2005), no. 6, 2009–2019.


\bibitem{gozzi_swiech_zhou_2010}
F. Gozzi, A. \'{S}wi\k{e}ch, X. Y. Zhou, \textit{Erratum: "A corrected proof of the stochastic verification theorem within the framework of viscosity solutions''.} SIAM J. Control Optim. 48 (2010), no. 6, 4177–4179.


\bibitem{grosset_viscolani}L. Grosset, B. Viscolani, \textit{Advertising for a new product introduction: A stochastic approach.} Top 12 (2004), 149–167.

\bibitem{krylov}
N. V. Krylov, \textit{Controlled diffusion processes}.  Science $\&$ Business Media, Vol. 14 (Springer, 2008).


  



    \bibitem{mendelson}
  E. Mendelson,  \textit{Introduction to mathematical logic.} Chapman and Hall/CRC, 2009.
  
      \bibitem{menoukeu_tangpi}
 O. Menoukeu-Pamen,  L. Tangpi, \textit{Maximum principle for stochastic control of SDEs with measurable drifts}, J. Optim. Theory Appl. 197 (2023), no. 3, 1195-1228.


\bibitem{nerlove_arrow} M. Nerlove, J. K. Arrow, \textit{Optimal advertising policy under dynamic conditions}. Economica 29 (1962), no. 114,  129-142.
  
  
      \bibitem{nisio}
   M. Nisio, \textit{Stochastic Control Theory: Dynamic Programming Principle}. Probability Theory and Stochastic Modelling, vol. 72 (Springer, Berlin, 2015)







\bibitem{koike_perron}
S. Koike, \textit{Perron’s method for $L^p$-viscosity solutions}. Saitama Math. J. 23 (2005), 9-28. 

\bibitem{pham_2009}
H. Pham, \textit{Continuous-time stochastic control and optimization with financial applications}. Stochastic Modelling and Applied Probability, vol. 61 (Springer, Berlin, 2009)







 
  \bibitem{stannat_wessels_2021}
W. Stannat, L. Wessels, \textit{Necessary and Sufficient Conditions for Optimal Control of Semilinear Stochastic Partial Differential Equations.} Ann. Appl. Probab. 34 (2024), no. 3, 3251-3287.
 
 \bibitem{swiech1997}
A.  \'{S}wi\k{e}ch, \textit{$W^{1,p}$-interior estimates for solutions of fully nonlinear, uniformly elliptic equations.} Adv. Differential Equations 2 (1997), no. 6, 1005–1027.
35J65 (35B45 35B65 49L25)

\bibitem{swiech2020}
A. \'{S}wi\k{e}ch, \textit{Pointwise properties of $L^p$-viscosity solutions of uniformly elliptic equations with quadratically growing gradient terms.} Discrete Contin. Dyn. Syst. 40 (2020), no. 5, 2945–2962.


\bibitem{touzi}
N. Touzi,  \textit{Optimal Stochastic Control, Stochastic Target Problems, and Backward SDE.} Fields
Institute Monographs. (Springer-Verlag, New York, 2012)


\bibitem{yong_zhou}
J. Yong, X. Y. Zhou, \textit{Stochastic Controls, Hamiltonian Systems and HJB Equations}. Applications of Mathematics, vol. 43 (Springer, New York, 1999)

\bibitem{zhang_zhu_zhu}
X. Zhang, R. Zhu, X.  Zhu, \textit{Singular HJB equations with applications to KPZ on the real line.} Probab. Theory Relat. Fields. 183 (2022) no. 3, 789-869.
Chicago	


\end{thebibliography}

\end{document}